\documentclass[secthm,seceqn,amsthm,ussrhead,12pt]{amsart}
\usepackage{amsmath,latexsym}
\usepackage[english]{babel}
\usepackage[psamsfonts]{amssymb}
\usepackage{times}
\usepackage{cite}

\usepackage[mathcal]{euscript}
\numberwithin{equation}{section} \textwidth=17.5cm
\newtheorem{theorem}{Theorem}[section]

\newtheorem{lemma}[theorem]{Lemma}

\newtheorem{problem}[theorem]{Problem}

\numberwithin{equation}{section}

\begin{document}

\title[2-Local derivations on  matrix algebras] {2-Local derivations
on  matrix algebras over semi-prime Banach  algebras and on $AW^\ast$-algebras }

\author[Shavkat Ayupov]{Shavkat Ayupov$^1$}

\address{%
$^1$Dormon yoli 29 \\
Institute of
 Mathematics  \\
 National University of
Uzbekistan\\
 100125  Tashkent\\   Uzbekistan\\
 and
 the Abdus Salam \\
 International Centre\\
 for Theoretical Physics (ICTP)\\
  Trieste, Italy}

\email{sh$_{-}$ayupov@mail.ru}

\author[Karimbergen Kudaybergenov]{Karimbergen Kudaybergenov$^2$}

\address{%
$^2$Ch. Abdirov 1 \\
Department of Mathematics\\
 Karakalpak state university \\
 Nukus
230113, Uzbekistan}

 \email{karim2006@mail.ru}

\begin{abstract} The paper is devoted to  2-local derivations on matrix algebras over unital
semi-prime Banach algebras. For     a unital
semi-prime Banach algebra $\mathcal{A}$ with the inner derivation property we
prove that any 2-local derivation on the algebra
$M_{2^n}(\mathcal{A}),$ $n\geq 2,$ is a derivation. We apply this
result to $AW^\ast$-algebras and show that any 2-local derivation
on an arbitrary $AW^\ast$-algebra is a derivation.

\\

\end{abstract}

\maketitle

\section{Introduction}

Given an algebra $\mathcal{A},$ a linear operator
$D:\mathcal{A}\rightarrow \mathcal{A}$ is called a
\textit{derivation}, if $D(xy)=D(x)y+xD(y)$ for all $x, y\in
\mathcal{A}$ (the Leibniz rule). Each element $a\in \mathcal{A}$
implements a derivation $D_a$ on $\mathcal{A}$ defined as
$D_a(x)=\mbox{ad}(a)(x)=ax-xa,$ $x\in \mathcal{A}.$ Such
derivations $D_a$ are said to be \textit{inner derivations}.

In 1997, P. Semrl \cite{Semrl97}  introduced the concepts of
2-local derivations and $2$-local automorphisms. Recall that a map
$\Delta:\mathcal{A}\rightarrow\mathcal{A}$  (not linear in
general) is called a
 $2$-\emph{local derivation} if  for every $x, y\in \mathcal{A},$  there exists
 a derivation $D_{x, y}:\mathcal{A}\rightarrow\mathcal{A}$
such that $\Delta(x)=D_{x, y}(x)$  and $\Delta(y)=D_{x, y}(y).$ In
particular, he has described $2$-local derivations on the algebra
$B(H)$ of all bounded linear operators on the infinite-dimensional
separable Hilbert space~$H.$ In \cite[Remark]{Semrl97}  P.
\v{S}emrl have wrote that the same results hold also in the case
that $H$ is finite-dimensional. In this case, however, he was only
able to get a long proof involving tedious computations, and so,
he  did not  include these results. In \cite{KimKim04} S.O.Kim and
J.S.Kim  gave a short proof that every 2-local derivation  on a
finite-dimensional complex matrix algebras is a derivation. The
methods of the proofs  above mentioned results from
~\cite{KimKim04} and \cite{Semrl97} are essentially based on the
fact that the algebra $B(H)$  is generated by two elements for
separable Hilbert space $H.$ Later J.~H.~Zhang and H.~X.~Li
\cite{Zhang} have extended the above mentioned result of
\cite{KimKim04} for arbitrary symmetric
 digraph matrix algebras  and constructed  an example of  $2$-local derivation
 which is not a derivation on  the algebra
of all upper triangular complex $2\times 2$-matrices.

As it was mentioned above, the proofs of the papers
\cite{KimKim04} and \cite{Semrl97}  are essentially based on the
fact that the algebra $B(H)$  is generated by two elements for
separable Hilbert space $H.$  Since the algebra $B(H)$  is  not
generated by two elements for  non separable $H$,  one cannot
directly apply the methods of the above papers in this case. In
\cite{AyuKuday2012} the authors suggested a new technique and have
generalized the above mentioned results of \cite{KimKim04} and
\cite{Semrl97} for arbitrary Hilbert spaces. Namely, we considered
$2$-local derivations on the algebra $B(H)$ of all bounded linear
operators on an arbitrary (no separability is assumed) Hilbert
space $H$ and proved that every $2$-local derivation on $B(H)$ is
a derivation. A similar result for $2$-local derivations on finite
von Neumann algebras was obtained in \cite{AKNA}. In \cite{AA2}
the authors extended all above results and give a short proof of
this result for arbitrary semi-finite von Neumann algebras.
Finally, in  \cite{AyuKuday2014}, using the analogue of Gleason
Theorem for signed measures, we have extended  this result to type
III von Neumann algebras. This  implies that on arbitrary
 von Neumann algebra each $2$-local derivation is a derivation.

In the present  paper we consider 2-local derivations on matrix
algebras over unital semi-prime Banach algebras. Let $\mathcal{A}$
be  a unital semi-prime Banach algebra with the inner derivation
property. We prove that any 2-local derivation on the algebra
$M_{2^n}(\mathcal{A}),$ $n\geq 2,$ is a derivation. We also apply
this result to $AW^\ast$-algebras and prove that any
2-local derivation on an arbitrary $AW^\ast$-algebra is a
derivation.

\medskip

\section{2-local derivations on matrix algebras}

\medskip

If  $\Delta :\mathcal{A}\rightarrow \mathcal{A}$ is  a 2-local
derivation, then from the definition it easily follows that
$\Delta$ is homogenous. At the same time,
\begin{equation}\label{joor}
 \Delta(x^2)=\Delta(x)x+x\Delta(x)
\end{equation}
for each $x\in \mathcal{A}.$

In \cite{Bre} it is proved that any Jordan derivation (i.e. a
linear map satisfying the above equation) on a semi-prime algebra
is a derivation. So, in the  case semi-prime algebras  in order to
prove that a 2-local derivation $\Delta : \mathcal{A}\rightarrow
\mathcal{A}$ is a derivation it is sufficient to prove that
$\Delta: \mathcal{A} \rightarrow \mathcal{A}$ is additive.

We say that an algebra $\mathcal{A}$ has the \textit{inner
derivation property} if every derivation on $\mathcal{A}$ is
inner. Recall that an algebra $\mathcal{A}$ is said to be
\textit{semi-prime} if $a\mathcal{A}a=0$ implies that $a=0.$

The following theorem is the main result of this section.

\begin{theorem}\label{mainlocal}
Let $\mathcal{A}$ be  a unital semi-prime Banach algebra with the
inner derivation property and let $M_{2^n}(\mathcal{A})$ be the
algebra of  $2^n \times 2^n$-matrices over $\mathcal{A}.$ Then any
2-local derivation $\Delta$ on $M_{2^n}(\mathcal{A})$ is a
derivation.
\end{theorem}

The proof  of Theorem~\ref{mainlocal} consists of two steps. In
the  first step  we shall show additivity of  $\Delta$ on the
 the subalgebra of diagonal matrices from $M_{2^n}(\mathcal{A}).$

Let $M_{n}(\mathcal{A})$ be the algebra of  $n \times n$-matrices
over $\mathcal{A}$ and let  $\{e_{i,j}\}_{i,j=1}^n$ be  the system
of matrix units in $M_n(\mathcal{A}).$ For $x \in
M_n(\mathcal{A})$ by $x_{i,j}$ we denote the $(i, j)$-entry of
$x,$ i.e. $x_{i,j}=e_{i,i}xe_{j,j},$ where $1 \leq i, j \leq n.$
We shall, when necessary, identify this element with the matrix from
$M_n(\mathcal{A})$ whose $(i,j)$-entry is $x_{i,j},$ other entries are
zero.

Further in Lemmata~\ref{idp}--\ref{lemmafive} we assume that
$n\geq 2.$

\begin{lemma}\label{idp}
Let $\mathcal{A}$ be  a unital  Banach algebra with  the inner
derivation property. Then the algebra $M_{n}(\mathcal{A})$ also
has the inner derivation property.
\end{lemma}

\begin{proof} Let $D$ be a derivation on $M_{n}(\mathcal{A}).$ Set
$$
a=\sum\limits_{i=1}^n D(e_{i,1})e_{1,i}.
$$

We have
$$
\sum\limits_{i=1}^n D(e_{i,1})e_{1,i}+\sum\limits_{i=1}^n
e_{i,1}D(e_{1,i})=\sum\limits_{i=1}^n
D(e_{i,1}e_{1,i})=\sum\limits_{i=1}^n
D(e_{i,i})=D\left(\sum\limits_{i=1}^n
e_{i,i}\right)=D(\textbf{1})=0,
$$
where $\textbf{1}$ is the unit matrix. Therefore
$a=-\sum\limits_{i=1}^n e_{i,1}D(e_{1,i}). $

From  the above equalities by direct calculations we obtain that
$$
ae_{k,s}-e_{k,s}a=D(e_{k,1})e_{1,s}+e_{k,1}D(e_{1,s})=D(e_{k,s}),
$$
i.e. $D(e_{k,s})=[a, e_{k,s}]$ for all $k,s\in \overline{1,n}.$

Put $D_0=D-\mbox{ad}(a).$ Since $D_0(e_{1,1})=0,$ it follows that
$D_0$ maps $e_{1,1}M_n(\mathcal{A})e_{1,1}\equiv \mathcal{A}$ into
itself. Therefore the restriction $D_0|_{\mathcal{A}}$ of $D_0$
onto $\mathcal{A}$ is a derivation. Since $\mathcal{A}$ has the
inner derivation property there exists an element $a_{1,1}\in
\mathcal{A}$ such that $D_0(x)=[a_{1,1}, x]$ for all $x\in
\mathcal{A}.$

Set $b=\sum\limits_{i=1}^ne_{i,1}a_{1,1}e_{1,i}.$ Let $x$ be a
matrix such that $x=e_{k,k}xe_{s,s}.$ Then
$$
bx-xb=e_{k,1}a_{1,1}e_{1,k}x-xe_{s,1}a_{1,1}e_{1,s}=e_{k,1}[a_{1,1},
e_{1,k}xe_{s,1}]e_{1,s}
$$
and
\begin{eqnarray*}
D_0(x) & = &
D_0(e_{k,1}e_{1,k}xe_{s,1}e_{1,s})=D_0(e_{k,1})e_{1,k}xe_{s,1}e_{1,s}+\\
& + &
e_{k,1}D_0(e_{1,k}xe_{s,1})e_{1,s}+e_{k,1}e_{1,k}xe_{s,1}D_0(e_{1,s})=\\
& = & 0+e_{k,1}D_0(e_{1,k}xe_{s,1})e_{1,s}+0=e_{k,1}[a_{1,1},
e_{1,k}xe_{s,1}]e_{1,s},
\end{eqnarray*}
i.e. $D_0(x)=[b, x]$ for all $x$ of the form $x=e_{k,k}xe_{s,s}.$
By linearity of $D_0$ we have that  $D_0=\mbox{ad}(b).$ So,
$D=\mbox{ad}(a+b).$ The proof is complete.
\end{proof}

Consider the following two matrices:
\begin{equation}\label{vv}
 u=\sum\limits_{i=1}^n \frac{1}{2^i}e_{i,i},\,
v=\sum\limits_{i=2}^n e_{i-1,i}.
\end{equation}

It is easy to see that an element $x \in M_n(\mathcal{A})$
commutes with $u$  if and only if it is diagonal, and if an
element $a$ commutes with $v,$ then $a$ is of the form
\begin{equation}\label{trian}
a=\left( \begin{array}{cccccc}
a_1 & a_2 & a_3 & \ldots & a_n \\
0 & a_1 &  a_2 & \ldots & a_{n-1}\\
0 & 0 &  a_1 & \ldots & a_{n-2}\\
\vdots& \vdots& \vdots &\vdots & \vdots\\
0 & 0 &\ldots & a_1 & a_2\\
0 & 0 &\ldots & 0 & a_1
\end{array} \right).
\end{equation}

A  result, similar to the following one,  was proved in \cite[Lemma
4.4]{AKA} for matrix algebras over commutative regular algebras.

\begin{lemma}\label{lemmatwo}
For every $2$-local derivation $\Delta$ on  $M_n(\mathcal{A}),$
there exists a derivation  $D$   such that
$\Delta|_{M_n(Z(\mathcal{A}))}=D|_{M_n(Z(\mathcal{A}))},$ where
$Z(\mathcal{A})$ is the center of the algebra $\mathcal{A}.$ In
particular,
$\Delta|_{\mbox{sp}\{e_{i,j}\}_{i,j=1}^n}=D|_{\mbox{sp}\{e_{i,j}\}_{i,j=1}^n},$
where $\mbox{sp}\{e_{i,j}\}_{i,j=1}^n$ is the linear span of the
set $\{e_{i,j}\}_{i,j=1}^n.$
\end{lemma}

\begin{proof}
By Lemma~\ref{idp} there exists an element  $a$ in
$M_n(\mathcal{A})$ such that
$$
\Delta(u)=[a, u],\, \Delta(v)=[a, v],
$$
where $u, v$ are the elements from~\eqref{vv}. Replacing $\Delta$
by $\Delta-\mbox{ad}(a)$, if necessary, we can assume that
$\Delta(u)=\Delta(v)=0.$

Let $i, j\in \overline{1, n}.$ Take   a matrix $h$ such that
$$
\Delta(e_{i,j})=[h, e_{i,j}],\,\Delta(u)=[h, u].
$$
Since $\Delta(u)=0,$  it follows that $h$ has   diagonal form,
i.e. $h=\sum\limits_{i=1}^n h_{i,i}.$ So we have
$$
\Delta(e_{i,j})=he_{i,j}-e_{i,j}h.
$$
In the same way, but starting with the element $v$ instead of $u$, we
obtain
$$
\Delta(e_{i,j})=be_{i,j}-e_{i,j}b,
$$
where  $b$  has the  form~\eqref{trian}, depending on $e_{i,j}.$
So
$$
\Delta(e_{i,j})=he_{i,j}-e_{i,j} h=b e_{i,j}-e_{i,j} b.
$$
Since
$$
he_{i,j} - e_{i,j} h= (h_{i,i}-h_{j,j})e_{i,j}
$$
 and
 $$
[b e_{i,j}- e_{i,j} b]_{i,j}=0,
$$
 it follows that $\Delta(e_{i,j})=0.$

Now let us take a matrix $x=\sum\limits_{i,j=1}^n
f_{i,j}e_{i,j}\in M_n(Z(\mathcal{A})).$ Then
\begin{eqnarray*}
e_{i,j}\Delta(x)e_{i,j} & = &
e_{i,j} D_{e_{i,j}, x} (x) e_{i,j} =\\
& = & D_{e_{i,j}, x} (e_{i,j} x e_{i,j})-D_{e_{i,j}, x}
(e_{i,j})x e_{i,j}-e_{i,j} x D_{e_{i,j}, x} (e_{i,j})=\\
&=& D_{e_{i,j}, x} (f_{j,i}  e_{i,j}) -0-0=f_{j,i} D_{e_{i,j}, x}
(e_{i,j}) =0,
\end{eqnarray*}
i.e. $ e_{i,j}\Delta(x)e_{i,j}=0$ for all $i,j\in \overline{1,n}.$
This means that $\Delta(x)=0.$ The proof is complete.
\end{proof}

Further in Lemmata~\ref{three}--\ref{adj} we assume that $\Delta$
is a 2-local derivation on the algebra $M_n(\mathcal{A}),$  such
that $\Delta|_{\mbox{sp}\{e_{i,j}\}_{i,j=1}^n}= 0.$

Let $\Delta_{i,j}$ be the restriction of $\Delta$ onto
$\mathcal{M}_{i,j}=e_{i,i}M_n(\mathcal{A})e_{j,j},$ where $1 \leq
i, j \leq n.$

\begin{lemma}\label{three}
$\Delta_{i,j}$ maps $\mathcal{M}_{i,j}$ into itself.
\end{lemma}

\begin{proof} First, let us show that
\begin{equation}\label{compo}
\Delta_{i,j}(x) =e_{i,i}  \Delta(x) e_{j,j}
\end{equation}
for all $x\in \mathcal{M}_{i,j}.$

Let $x=x_{i,j}\in \mathcal{M}_{i,j}.$ Take a derivation $D$ on
$M_n(\mathcal{A})$ such that
$$
\Delta(x)=D(x),\, \Delta(e_{i,j})=D(e_{i,j}).
$$

It is suffices to  consider the following two cases.

Case 1. Let $i=j.$ Then
\begin{eqnarray*}
\Delta_{i,i}(x) & = &
\Delta(x)=D(x)=D(e_{i,i} x e_{i,i}) =\\
& = & D(e_{i,i}) x e_{i,i}+e_{i,i}D(x)e_{i,i}+e_{i,i} x D(e_{i,i})=\\
&=& 0+e_{i,i} \Delta(x) e_{i,i}+0=e_{i,i}  \Delta(x) e_{i,i}.
\end{eqnarray*}

Case 2. Let $i\neq j.$ Denote by  $\mathbf{1}$ the unit matrix.
Since $e_{i,i}x(\mathbf{1}-e_{i,i})=x$ and
\linebreak $(\mathbf{1}-e_{j,j})xe_{j,j}=x,$ we obtain that
\begin{eqnarray*}
e_{i,i}\Delta(x)(\mathbf{1}-e_{i,i}) & = &
e_{i,i} D_{e_{i,i}, x} (x)(\mathbf{1}-e_{i,i}) =\\
& = & D_{e_{i,i}, x} (e_{i,i} x (\mathbf{1}-e_{i,i}))-D_{e_{i,i},
x} (e_{i,i}) x (\mathbf{1}-e_{i,i})-
e_{i,i} x D_{e_{i,i}, x} (\mathbf{1}-e_{i,i})=\\
&=& D_{e_{i,i}, x} (x)-0-0=\Delta(x)
\end{eqnarray*}
and
\begin{eqnarray*}
(\mathbf{1}-e_{j,j})\Delta(x)e_{j,j} & = &
(\mathbf{1}-e_{j,j}) D_{e_{j,j}, x} (x)e_{j,j} =\\
& = & D_{e_{j,j}, x} ((\mathbf{1}-e_{j,j}) x e_{j,j})-D_{e_{j,j},
x} (\mathbf{1}-e_{j,j}) x
e_{j,j}-(\mathbf{1}-e_{j,j}) x D_{e_{j,j}, x} (e_{j,j})=\\
&=& D_{e_{j,j}, x} (x)-0-0=\Delta(x).
\end{eqnarray*}
Hence
\begin{eqnarray*}
e_{i,i}\Delta(x)e_{j,j} & = & (\mathbf{1}-e_{j,j})
e_{i,i}\Delta(x)(\mathbf{1}-e_{i,i})e_{j,j} =
(\mathbf{1}-e_{j,j})\Delta(x)e_{j,j}=\Delta(x).
\end{eqnarray*}
The proof is complete. \end{proof}

\begin{lemma}\label{lemmafour} Let
$x=\sum\limits_{i=1}^n x_{i,i}$ be a diagonal matrix. Then
$e_{k,k}\Delta(x)e_{k,k}=\Delta(x_{k,k})$ for all $k\in
\overline{1,n}.$
\end{lemma}

\begin{proof} Take an element $a$ from $M_n(\mathcal{A})$ such that
$$
\Delta(x)=[a, x],\, \Delta(x_{k,k})=[a, x_{k,k}].
$$
Since  $x$ is a diagonal matrix, the equality
\eqref{compo} implies that
\begin{eqnarray*}
\Delta(x_{k,k}) & = & e_{k,k}\Delta(x_{k,k})e_{k,k}=e_{k,k}[a,
x_{k,k}] e_{k,k} = [a_{k,k},  x_{k,k}]
\end{eqnarray*}
and
\begin{eqnarray*}
e_{k,k}\Delta(x)e_{k,k} & = & e_{k,k}[a, x] e_{k,k} = [a_{k,k},
x_{k,k}].
\end{eqnarray*}
Thus $e_{k,k}\Delta(x)e_{k,k}=\Delta(x_{k,k}).$ The proof is
complete. \end{proof}

\begin{lemma}\label{lemmafive}
 $e_{j,i}\Delta_{i,i}(x)e_{i,j}=\Delta_{j,j}(e_{j,i}xe_{i,j})$ for all
 $x=x_{i,i}\in \mathcal{M}_{i,i}.$
\end{lemma}

\begin{proof} For $i=j$ we have already proved (see Lemma~\ref{lemmafour}).

Suppose that $i\neq j.$  For arbitrary element $x=x_{i,i}\in
\mathcal{M}_{i,i}$ , consider  $y=x+e_{j,i}xe_{i,j}\in
\mathcal{M}_{i,i}+\mathcal{M}_{j,j}.$ Take an element $a\in
\mathcal{A}$ such that
\begin{center}
$\Delta(x)=[a, y]$ and $\Delta(v)=[a, v],$
\end{center}
where $v$ is the element from~\eqref{vv}. Since $\Delta(v)=0,$ it
follows that $a$ has the form~\eqref{trian}. By
Lemma~\ref{lemmafour} we obtain that
$$
e_{j,i}\Delta_{i,i}(x)e_{i,j}=e_{j,i}e_{i,i}\Delta(y)e_{i,i}e_{i,j}=e_{j,i}[a,
y]e_{i,j}=e_{j,i}[a_1, x]e_{i,j}
$$
and
$$
\Delta_{j,j}(e_{j,i}xe_{i,j})=e_{j,j}\Delta(y)e_{j,j}=e_{j,j}[a,
y]e_{j,j}=e_{j,j}[a,x+e_{j,i}xe_{i,j}]e_{j,j}=e_{j,i}[a_1,
x]e_{i,j}.
$$
The proof is complete. \end{proof}

Further in Lemmata~\ref{six}--\ref{adj} we assume that $n\geq 3.$

\begin{lemma}\label{six}
 $\Delta_{i,i}$ is additive for all $i\in \overline{1,n}.$
\end{lemma}

\begin{proof} Let $i\in \overline{1,n}.$ Since $n\geq 3,$ we can take different numbers $k, s$ such that $k\neq i, s\neq i.$

For arbitrary $x, y\in \mathcal{M}_{i,i}$  consider
$z=x+y+e_{k,i}xe_{i,k}+e_{s,i}ye_{i,s}\in
\mathcal{M}_{i,i}+\mathcal{M}_{k,k}+\mathcal{M}_{s,s}.$ Take a
matrix  $a\in M_n(\mathcal{A})$ such that
\begin{center}
$\Delta(z)=[a, z]$ and $\Delta(v)=[a, v],$
\end{center}
where $v$ is  the element from~\eqref{vv}. Since $\Delta(v)=0,$ it
follows that $a$ has the form~\eqref{trian}. Using
Lemma~\ref{lemmafour} we obtain that
\begin{eqnarray*}
& & \Delta_{i,i}(x+y) = e_{i,i}\Delta(z)e_{i,i}=e_{i,i}[a,
z]e_{i,i}=[a_1, x+y],\\
& & \Delta_{i,i}(x) = e_{i,k}\Delta(e_{k,i}x
e_{i,k})e_{k,i}=e_{i,k}e_{k,k}\Delta(z)e_{k,k}e_{k,i}=e_{i,k}[a,
z]e_{k,i}=[a_1, x],\\
& & \Delta_{i,i}(y)=e_{i,s}\Delta(e_{s,i}y
e_{i,s})e_{s,i}=e_{i,s}e_{s,s}\Delta(z)e_{s,s}e_{s,i}=e_{i,s}[a,
z]e_{s,i}=[a_1, y].
\end{eqnarray*} Hence
$$
\Delta_{i,i}(x+y)=\Delta_{i,i}(x)+\Delta_{i,i}(y).
$$
The proof is complete. \end{proof}

As it was mentioned in the beginning of the section  any additive
2-local derivation on a semi-prime algebra is a derivation. Since
$\mathcal{M}_{i,i}\equiv \mathcal{A}$ is semi-prime,
Lemma~\ref{six} implies the following result.

\begin{lemma}\label{lemmaseven}
 $\Delta_{i,i}$ is a derivation for all $i\in \overline{1,n}.$
\end{lemma}

Since $\Delta_{1,1}$ is a derivation on
$e_{1,1}M_n(\mathcal{A})e_{1,1}\equiv \mathcal{A}$ and
$\mathcal{A}$ has the inner derivation property, it follows that
there exists an element $a_{1,1}$ in $\mathcal{A}$ such that
$\Delta_{1,1}=\mbox{ad}(a_{1,1}).$ Set
$\tilde{a}=\sum\limits_{i=1}^n e_{i,1}a_{1,1}e_{1,i}.$

Denote by $\mathcal{D}_n$ the set of all diagonal matrices from $M_n(\mathcal{A}),$ i.e.
the set of all matrices of the following form
\[
x=\left( \begin{array}{cccccc}
x_1 & 0 & 0 & \ldots & 0 \\
0 & x_2 &  0 & \ldots & 0\\
\vdots& \vdots& \vdots &\vdots & \vdots\\
0 & 0 &\ldots & x_{n-1} & 0\\
0 & 0 &\ldots & 0 & x_n
\end{array} \right).
\]

Let $x\in \mathcal{D}_n.$ Then
\begin{eqnarray*}
[\tilde{a},x]_{i,i} & = & e_{i,i}\tilde{a}x e_{i,i}-e_{i,i}x
\tilde{a}e_{i,i}=e_{i,1}a_{1,1}e_{1,i}x_{i,i}-x_{i,i}e_{1,i}a_{1,1}e_{1,i}=\\
&=& e_{i,1}[a_{1,1}, e_{1,i}x_{i,i}
e_{i,1}]e_{1,i}=e_{i,1}\Delta_{1,1}(e_{1,i}x_{i,i}
e_{i,1})e_{1,i}=\Delta_{i,i}(x),
\end{eqnarray*}
i.e. $\mbox{ad}(\tilde{a})|_{\mathcal{D}_n}
=\Delta|_{\mathcal{D}_n}.$

Further,
\begin{eqnarray*}
[\tilde{a},e_{i,j}] & = & \tilde{a} e_{i,j}-e_{i,j}
\tilde{a}=e_{i,1}a_{1,1}e_{1,i}e_{i,j}-e_{i,j}e_{j,1}a_{1,1}e_{1,j}=\\
&=& e_{i,1}a_{1,1}e_{1,j}-e_{i,1}a_{1,1}e_{1,j}=0,
\end{eqnarray*}
i.e. $\mbox{ad}(\tilde{a})|_{\mbox{sp}\{e_{i,j}\}_{i,j=1}^n}\equiv
0,$ where $\mbox{sp}\{e_{i,j}\}_{i,j=1}^n$ is the linear span of
the set $\{e_{i,j}\}_{i,j=1}^n.$

So, we have proved the following result.

\begin{lemma}\label{adj}
$\Delta|_{\mathcal{D}_n}=\mbox{ad}(\tilde{a})|_{\mathcal{D}_n}$
and $\mbox{ad}(\tilde{a})|_{\mbox{sp}\{e_{i,j}\}_{i,j=1}^n}=0.$
\end{lemma}

Replacing, if necessary,  $\Delta$ by $\Delta-\mbox{ad}(\tilde{a}),$ below in this
section we assume that
$$
\Delta|_{\mathcal{D}_n}\equiv 0,\,
\Delta|_{\mbox{sp}\{e_{i,j}\}_{i,j=1}^n}\equiv 0.
$$

Now we are in position to pass to the second step of our proof. As
the second  step let us show that if a 2-local derivation
$\Delta$ on a matrix algebra  equals to zero on   all diagonal  matrices
and on the linear span  of matrix units, then it
is identically zero on the whole algebra. In order to prove this we
first consider the $2\times 2$-matrix algebras case.

\subsection{ The case of $\textbf{2}\times \textbf{2}$-matrices}

In this subsection we shall assume that $\mathcal{B}$ is a unital
Banach algebra with the inner derivation property and $\Delta$ is
a 2-local derivation on $M_2(\mathcal{B})$, such that
\begin{center}
$\Delta|_{\mbox{sp}\{e_{i,j}\}_{i,j=1}^2}\equiv0$ and
$\Delta|_{\mathcal{D}_2}\equiv 0.$
\end{center}
We also denote by $e$ the unit of the algebra $\mathcal{B}.$

\begin{lemma}\label{ss}
$\Delta(x)_{1,1}=\Delta(x)_{2,2}=0$ for all $x\in
M_2(\mathcal{B}).$
\end{lemma}

\begin{proof} Let $x=\left( \begin{array}{cc}
x_{1,1} & x_{1,2} \\
x_{2,1} & x_{2,2}
\end{array} \right).$ Since
$\Delta$ is homogeneous, we can assume that $\|x_{1,1}\|<1,$ where
$\|\cdot\|$ is the norm on $\mathcal{B}.$ Set $y=\left(
\begin{array}{cc}
e+x_{1,1} & 0 \\
0 & 0
\end{array} \right).$ Since $\|x_{1,1}\|<1,$ it
follows that $e+x_{1,1}$ is invertible in $\mathcal{B}.$ Take an
element $a\in M_2(\mathcal{B})$ such that
$$
\Delta(x)=[a,x],\, \Delta(y)=[a ,y].
$$
Since  $y\in \mathcal{D}_2$ we have that $0=\Delta(y)=[a ,y],$ and
therefore
$$
0=\Delta(y)_{1,1}=a_{1,1}(e+x_{1,1})-(e+x_{1,1})a_{1,1}=0,
$$
$$
0=\Delta(y)_{2,1}=a_{2,1}(e +x_{1,1})=0,
$$
and
$$
0=\Delta(y)_{1,2}=-(e +x_{1,1}) a_{1,2}=0.
$$
Thus
$$
a_{1,1}x_{1,1}-x_{1,1}a_{1,1}=0
$$
and
$$
a_{2,1}=a_{1,2}=0.
$$
The above equalities imply that
$$
\Delta(x)_{1,1}=a_{1,1}x_{1,1}-x_{1,1}a_{1,1}=0.
$$
In a similar way we can show that $\Delta(x)_{2,2}=0.$  The proof
is complete.
\end{proof}

\begin{lemma}\label{zerro}
Let   $x$ be a matrix with $x_{k,s}=\lambda e,$ where $\lambda\in
\mathbb{C}.$ Then $\Delta(x)_{k,s}=0.$
\end{lemma}

\begin{proof}
\begin{eqnarray*}
e_{s,k}\Delta(x)e_{s,k} & = &
e_{s,k}D_{e_{s,k}, x} (x) e_{s,k}  =\\
& = & D_{e_{s,k}, x} (e_{s,k} x e_{s,k})-D_{e_{s,k}, x}(e_{s,k})
x e_{s,k}-e_{s,k} x D_{e_{s,k}, x} (e_{s,k})=\\
&=& \lambda D_{e_{s,k}, x}(e_{s,k})-0-0=0.
\end{eqnarray*}
Thus
$$
\Delta(x)_{k,s}=e_{k,k}\Delta(x)e_{s,s}=e_{k,s}e_{s,k}\Delta(x)e_{s,k}e_{k,s}=0.
$$
The proof is complete. \end{proof}

\begin{lemma}\label{cc} Let $x=\left( \begin{array}{cc}
x_{1,1} & x_{1,2} \\
x_{2,1} & x_{2,2}
\end{array} \right)$ and
$y=\left( \begin{array}{cc}
x_{1,1} & x_{1,2} \\
0 & x_{2,2}
\end{array} \right).$ Then
$\Delta(x)_{1,2}=\Delta(y)_{1,2}.$
\end{lemma}

\begin{proof} Take a
matrix  $a\in M_2(\mathcal{B})$ such that
\begin{center}
$\Delta(x)=[a, x]$ and $\Delta(y)=[a, y].$
\end{center}
Then
$$
\Delta(x)_{1,2}=a_{1,1}x_{1,2}+a_{1,2}x_{2,2}-x_{1,1}a_{1,2}-x_{1,2}a_{2,2}=\Delta(y)_{1,2}.
$$
The proof is complete.
\end{proof}

\begin{lemma}\label{z} Let
$x=\left( \begin{array}{cc}
e+x_{1,1} & x_{1,1} \\
\lambda e & 0
\end{array} \right),$ where
$x_{1,1}\in \mathcal{B},$ $\|x_{1,1}\|<1,$ $\lambda\in
\mathbb{C}.$ Then $\Delta(x)=0.$
\end{lemma}

\begin{proof} From Lemmata~\ref{ss} and \ref{zerro}, it follows that
$$
\Delta(x)_{1,1}=\Delta(x)_{2,2}=\Delta(x)_{2,1}=0.
$$

Let us to show that $\Delta(x)_{1,2}=0.$

Case 1. Let $\lambda=0.$ Take a matrix  $a\in M_2(\mathcal{B})$
such that
\begin{center}
$\Delta(x)=[a, x]$ and $\Delta(e_{2,1})=[a, e_{2,1}].$
\end{center}
Since the element $a$ commutes with $e_{2,1},$ it follows that $a$
is of the form $a=\left( \begin{array}{cc}
a_1 & 0 \\
a_2 & a_1
\end{array} \right).
$ Then
$$
0=\Delta(x)_{2,1}=a_2(e+x_{1,1})=0.
$$
Since $e+x_{1,1}$ is invertible in $\mathcal{B},$ it follows that
$a_2=0.$ From the last equality we obtain that
$$
0=\Delta(x)_{1,1}=a_1x_{1,1}-x_{1,1}a_1=\Delta(x)_{1,2}.
$$
i.e. $\Delta(x)_{1,2}=0.$ Therefore, $\Delta(x)=0.$

Case 2. Let $\lambda\neq 0.$ Set $y=\left( \begin{array}{cc}
e+x_{1,1} & x_{1,1} \\
0 & 0
\end{array} \right).$
By Case 1,  $\Delta(y)=0.$ Applying Lemma~\ref{cc} we obtain that
$$
0=\Delta(y)_{1,2}=\Delta(x)_{1,2}.
$$
i.e. $\Delta(x)_{1,2}=0.$ Thus $\Delta(x)=0.$ The proof is
complete.
\end{proof}

\begin{lemma}\label{zz}
If $x=\left( \begin{array}{cc}
x_{1,1} & x_{1,2} \\
\lambda e & 0
\end{array} \right),$
where $x_{1,1}$ is an invertible element in $\mathcal{B},$  then
$\Delta(x)=0.$
\end{lemma}

\begin{proof} As in the proof of Lemma~\ref{z} it is suffices to show that
$\Delta(x)_{1,2}=0.$  Since $\Delta$ is homogeneous, we can assume
that $\|x_{1,2}\|<1.$

Case 1. Let $\lambda=0.$ Set $y=\left( \begin{array}{cc}
e+x_{1,2} & x_{1,2} \\
e & 0
\end{array} \right).$
Take a matrix $a\in M_2(\mathcal{B})$ such that
\begin{center}
$\Delta(x)=[a, x]$ and $\Delta(y)=[a, y].$
\end{center}
By Lemma~\ref{z} we have that $\Delta(y)=0.$ Since
$$
0=\Delta(x)_{2,1}=a_{2,1} x_{1,1},
$$
and $x_{1,1}$ is invertible in $\mathcal{B},$ it follows that
$a_{2,1}=0.$ From
$$
0=\Delta(y)_{2,2}=a_{2,1}x_{1,2}-a_{1,2},
$$
it follows that $a_{1,2}=0.$ So, $a$ is a diagonal matrix.
This implies that
$$
0=\Delta(y)_{1,2}=a_{1,1}x_{1,2}-x_{1,2}a_{2,2}=\Delta(x)_{1,2},
$$
i.e. $\Delta(x)_{1,2}=0.$ Therefore $\Delta(x)=0.$

Case 2. Let $\lambda\neq 0.$ Set $y=\left( \begin{array}{cc}
x_{1,1} & x_{1,2} \\
0 & 0
\end{array} \right).$
By Case 1,  $\Delta(y)=0.$ From Lemma~\ref{cc} we obtain that
$$
0=\Delta(y)_{1,2}=\Delta(x)_{1,2}.
$$
i.e. $\Delta(x)_{1,2}=0$ , and hence $\Delta(x)=0.$ The proof is
complete.
\end{proof}

\begin{lemma}\label{zzz}
Let $x=\left( \begin{array}{cc}
x_{1,1} & x_{1,2} \\
0 & x_{2,2}
\end{array} \right),$ where $x_{1,1}$ is an invertible element in $\mathcal{B}.$
Then $\Delta(x)=0.$
\end{lemma}

\begin{proof} Set $y=\left( \begin{array}{cc}
x_{1,1} & x_{1,2} \\
e & 0
\end{array} \right).$
Take a matrix  $a\in M_2(\mathcal{B})$ such that
\begin{center}
$\Delta(x)=[a, x]$ and $\Delta(y)=[a, y].$
\end{center}
By Lemma~\ref{zz} we have that $\Delta(y)=0.$ Since
$$
0=\Delta(y)_{1,1}-\Delta(x)_{1,1}=a_{1,2},
$$
it follows that $a_{1,2}=0.$  From the last
equality  we obtain that
$$
0=\Delta(y)_{1,2}=a_{1,1}x_{1,2}-x_{1,2}a_{2,2}=\Delta(x)_{1,2},
$$
i.e. $\Delta(x)_{1,2}=0.$  Therefore $\Delta(x)=0.$ The proof is
complete.
\end{proof}

\begin{lemma}\label{mmm} Let $\Delta$ be a 2-local derivation on
$M_2(\mathcal{B})$  such that $$
\Delta|_{\mbox{sp}\{e_{i,j}\}_{i,j=1}^2}\equiv
0,\,\Delta|_{\mathcal{D}_2}\equiv 0.
$$
Then  $\Delta\equiv 0.$
\end{lemma}

\begin{proof}
Let $x=\left( \begin{array}{cc}
x_{1,1} & x_{1,2} \\
x_{2,1} & x_{2,2}
\end{array} \right).$
By Lemma~\ref{ss}, $\Delta(x)_{1,1}=\Delta(x)_{2,2}=0.$

Let us to show that $\Delta(x)_{1,2}=0.$  Since $\Delta$ is
homogeneous we can assume that $\|x_{1,1}\|<1.$ Further, since
  $\Delta(x)=\Delta(\mathbf{1}+x),$ replacing, if necessary,
 $x$ by $\mathbf{1}+x$, we may assume that $x_{1,1}$ is
invertible in $\mathcal{B}.$

Put $y=\left( \begin{array}{cc}
x_{1,1} & x_{1,2} \\
0 & x_{2,2}
\end{array} \right).$
Lemma~\ref{zzz} implies that $\Delta(y)=0.$ Now from
Lemma~\ref{cc} we obtain that
$$
\Delta(x)_{1,2}=\Delta(y)_{1,2}=0.
$$

In a similar way we can show that $\Delta(x)_{2,1}=0.$  The proof
is complete.
\end{proof}

\subsection{The general case}

Now we are in position to prove
Theorem~\ref{mainlocal}.

\begin{proof}[\textit{Proof of Theorem~\ref{mainlocal}}]
Let $\Delta$ be a 2-local derivation on $M_{2^n}(\mathcal{A}),$
where $n\geq 2.$ By Lemma~\ref{lemmatwo}   there exists a
derivation $D$ on $M_{2^n}(\mathcal{A})$ such that
$\Delta|_{\mbox{sp}\{e_{i,j}\}_{i,j=1}^{2^n}}=D|_{\mbox{sp}\{e_{i,j}\}_{i,j=1}^{2^n}}.$
 Replacing, if necessary, $\Delta$ by $\Delta-D,$ we may assume that
$\Delta$ is equal to zero on $\mbox{sp}\{e_{i,j}\}_{i,j=1}^{2^n}.$
Further, by Lemma~\ref{adj}  there exists a diagonal element
$\tilde{a}$ in $M_{2^n}(\mathcal{A})$ such that
$\Delta|_{\mathcal{D}_{2^n}}=\mbox{ad}(\tilde{a})|_{\mathcal{D}_{2^n}}.$
Now replacing $\Delta$ by $\Delta-\mbox{ad}(\tilde{a}),$ we can
assume that $\Delta$ is identically zero on $\mathcal{D}_{2^n}.$ So, we
can assume that
\begin{center}
$\Delta|_{\mbox{sp}\{e_{i,j}\}_{i,j=1}^{2^n}}\equiv 0$ and
$\Delta|_{\mathcal{D}_{2^n}}\equiv 0.$
\end{center}

Let us to show that $\Delta\equiv 0.$ We proceed by induction on
$n.$

Let $n=2.$ We identify  the algebra $M_4(\mathcal{A})$ with the
algebra of $2\times 2$-matrices   $M_2(\mathcal{B}),$ over
$\mathcal{B}=M_2(\mathcal{A}).$

Let $\{e_{i,j}\}_{i,j=1}^4$ be a system of matrix units in
$M_4(\mathcal{A}).$ Then
$$
p_{1,1}=e_{1,1}+e_{2,2},\, p_{2,2}=e_{3,3}+e_{4,4},\,
p_{1,2}=e_{1,3}+e_{2,4},\, p_{2,1}=e_{3,1}+e_{4,2}
$$
is the system of matrix units in $M_2(\mathcal{B}).$ Since
$\Delta|_{\mbox{sp}\{e_{i,j}\}_{i,j=1}^{4}}\equiv 0,$ it follows
that $\Delta|_{\mbox{sp}\{p_{i,j}\}_{i,j=1}^{2}}\equiv 0.$

 Take an arbitrary element $x\in
p_{1,1}M_2(\mathcal{B})p_{1,1}\equiv \mathcal{B}.$ Choose a
derivation $D$ on $M_2(\mathcal{B})$ such that
$$
\Delta(x)=D(x),\, \Delta(p_{1,1})=D(p_{1,1}).
$$
Since  $\Delta(p_{1,1})=0,$ we obtain that
$$
p_{1,1}\Delta(x)p_{1,1}=p_{1,1}D(x)p_{1,1}=D(p_{1,1}
xp_{1,1})-D(p_{1,1})xp_{1,1}-p_{1,1}x D(p_{1,1})=\Delta(x).
$$
This means that the restriction $\Delta_{1,1}$ of $\Delta$ onto
$p_{1,1}M_2(\mathcal{B})p_{1,1}\equiv \mathcal{B}$ maps
$\mathcal{B}=M_2(\mathcal{A})$ into itself.

If $\mathcal{D}_4$ is the subalgebra of diagonal matrices from $M_4(\mathcal{A}),$ then
$p_{1,1}\mathcal{D}_4p_{1,1}$ is the  subalgebra  of diagonal matrices in the algebra
$M_2(\mathcal{A}).$ Since $\Delta|_{\mathcal{D}_{4}}\equiv 0,$ it
follows that $\Delta_{1,1}$ is identically zero on  the diagonal matrices
from  $M_{2}(\mathcal{A}).$ So,
\begin{center}
$\Delta_{1,1}|_{\mbox{sp}\{e_{i,j}\}_{i,j=1}^{2}}\equiv 0$ and
$\Delta_{1,1}|_{p_{1,1}\mathcal{D}_4p_{1,1}}\equiv 0.$
\end{center}
By Lemma~\ref{mmm} it follows that $\Delta_{1,1}\equiv 0.$

Let $\mathcal{D}_2$ be the set of diagonal matrices from  $M_2(\mathcal{B}).$ Since
$$
\mathcal{D}_2=\left( \begin{array}{cc}
\mathcal{B} & 0 \\
0 & \mathcal{B}
\end{array} \right)
$$
and $\Delta_{1,1}=0,$ Lemma~\ref{lemmafour} implies  that
$\Delta|_{\mathcal{D}_2}\equiv 0.$ Hence, $\Delta$ is a 2-local
derivation on $M_2(\mathcal{B})$ such that
\begin{center}
$\Delta|_{\mbox{sp}\{p_{i,j}\}_{i,j=1}^{2}}\equiv 0$ and
$\Delta|_{\mathcal{D}_2}\equiv 0.$
\end{center}
Again by Lemma~\ref{mmm} it
follows that $\Delta\equiv 0.$

Now assume that the assertion of the Theorem is true for $n-1.$

Considering the algebra   $M_{2^n}(\mathcal{A})$ as the
algebra of $2\times 2$-matrices  $M_2(\mathcal{B})$ over
$\mathcal{B}=M_{2^{n-1}}(\mathcal{A})$ and repeating the above
arguments we obtain that  $\Delta\equiv 0.$ The proof is complete.
\end{proof}

The condition on the algebra $\mathcal{A}$ to be a Banach algebra
was applied only for the invertibility of elements of the forms
$\mathbf{1}+x,$ where $x\in \mathcal{A},\, \|x\|<1.$  In this
connection the following question naturally arises.

\begin{problem}
Does  Theorem~\ref{mainlocal} hold for arbitrary (not necessarily normed)
 algebra $\mathcal{A}$ with the inner derivation property?
\end{problem}

\section{An application to $AW^\ast$-algebras}

In this section we apply Theorem~\ref{mainlocal} to the
description of 2-local derivations on $AW^\ast$-algebras.

\begin{theorem}\label{kaplanal}
Let $\mathcal{A}$ be an arbitrary $AW^\ast$-algebra. Then any
2-local derivation $\Delta$ on $\mathcal{A}$ is a derivation.
\end{theorem}

\begin{proof}
Let us first note that any $AW^\ast$-algebra is semi-prime.
It is also known \cite{Ole} that $AW^\ast$-algebra has the inner
derivation property.

Let $z$ be a central projection in $\mathcal{A}.$ Since $D(z)=0$
for an arbitrary derivation $D,$ it is clear that $\Delta(z)=0$
for any $2$-local derivation $\Delta$ on $\mathcal{A}.$ Take $x\in
\mathcal{A}$ and let $D$ be a derivation on $\mathcal{A}$ such
that $\Delta(z x)=D(z x), \Delta(x)=D(x)$. Then we have $\Delta(z
x)=D(z x)=D(z)x+zD(x)=z\Delta(x).$ This means that every 2-local
derivation $\Delta$ maps $z\mathcal{A}$ into $z\mathcal{A}$ for
each central projection $z\in \mathcal{A}.$ So, we may consider
the restriction of $\Delta$ onto $e\mathcal{A}.$ Since an
arbitrary  $AW^\ast$-algebra can be decomposed along a central
projection into the direct sum of an abelian $AW^\ast$-algebra,  and $AW^\ast$-algebras
 of  type I$_n,n\geq2,$ type I$_\infty,$, type II and type III , we may consider these cases separately.

Let $\mathcal{A}$ be an abelian $AW^\ast$-algebra. It is
well-known that any derivation on a such algebra is identically
zero. Therefore any 2-local derivation on an abelian
$AW^\ast$-algebra is also identically zero.

If  $\mathcal{A}$ is an  $AW^\ast$-algebra of  type I$_n,$
$n\geq2,$ with the center $Z(\mathcal{A}),$ then it is isomorphic
to the algebra $M_n(Z(\mathcal{A})).$ By Lemma~\ref{lemmatwo}
there exists a derivation $D$ on $\mathcal{A}\equiv
M_n(Z(\mathcal{A}))$ such that $\Delta\equiv D.$ So,  $\Delta$ is
a derivation.

Let  the $AW^\ast$-algebra $\mathcal{A}$ have one of the types
I$_\infty,$ II or III. Then  the halving Lemma \cite[Lemma
4.5]{Kap51} for type I$_\infty$-algebras and \cite[Lemma
4.12]{Kap51} for type II or III algebras, imply that the unit of
the algebra $\mathcal{A}$ can be represented as a sum of mutually
equivalent orthogonal projections $e_1, e_2, e_3, e_4$ from
$\mathcal{A}.$ Then the map $x\mapsto \sum\limits_{i,j=1}^4
e_ixe_j$ defines an  isomorphism between the algebra $\mathcal{A}$ and
the matrix algebra $M_4(\mathcal{B}),$ where
$\mathcal{B}=e_{1,1}\mathcal{A}e_{1,1}.$ Therefore
Theorem~\ref{mainlocal} implies that any 2-local derivation on
$\mathcal{A}$ is a derivation. The proof is complete.
\end{proof}

\end{document}